\documentclass[12pt,a4paper,oneside]{amsart}
\usepackage{amsfonts, amsmath, amssymb, amsthm}
\usepackage[colorlinks=true,citecolor=blue]{hyperref}
\usepackage[margin=1.4in]{geometry}
\usepackage{times}
\usepackage[T1]{fontenc}
\usepackage{bbm}
\usepackage{mathtools}

\usepackage{graphicx}
\usepackage{enumerate}
\usepackage{verbatim}
\usepackage{tikz}
\usepackage{algorithmic}

\newtheorem{theorem}{Theorem}[section]
\newtheorem*{theorem*}{Theorem}
\newtheorem{lemma}[theorem]{Lemma}

\theoremstyle{definition}

\newtheorem*{definition*}{Definition}

\newtheorem*{remark*}{Remark}

\begin{document}

\author{Andreas F. Holmsen}

\date{\today}

 \address{Andreas F. Holmsen, \hfill \hfill \linebreak 
 Department of Mathematical Sciences,  \hfill \hfill \linebreak
 KAIST, 
 Daejeon, South Korea.  \hfill \hfill }
 \email{andreash@kaist.edu}

\title{Large cliques in hypergraphs with forbidden substructures}

\begin{abstract} 
A result due to Gy{\'a}rf{\'a}s, Hubenko, and Solymosi (answering a question of Erd{\H o}s) states that if a graph $G$ on $n$ vertices does not contain $K_{2,2}$ as an induced subgraph yet has at least $c\binom{n}{2}$ edges, then $G$ has a complete subgraph on at least $\frac{c^2}{10}n$ vertices. In this paper we suggest a ``higher-dimensional'' analogue of the notion of an induced $K_{2,2}$ which allows us to generalize their result to $k$-uniform hypergraphs. Our result also has an interesting consequence in discrete geometry. In particular, it implies that the {\em fractional Helly theorem} can be derived as a purely combinatorial consequence of the {\em colorful Helly theorem}.
\end{abstract}

\maketitle 

\section{Introduction}

Among the classical problems in extremal graph theory are the {\em Tur{\'a}n type extremal problems}. They ask for the maximum number of edges $\mbox{ex}(n,H)$ in a graph on $n$ vertices provided it does not contain some fixed graph $H$ as a subgraph. In the case when $H = K_m$, the complete graph on $m$ vertices, the answer is given by Turan's theorem \cite{turan}, which also characterizes the extremal graphs which obtain the maximum $\mbox{ex}(n, K_m)$ for all $n$ and $m$. More generally, if the chromatic number $\chi(H)\geq 3$, we have $\mbox{ex}(n,H) = (1-\frac{1}{\chi(H)-1})\binom{n}{2}+o(n^2)$. This is the fundamental Erd{\H o}s--Stone--Simonovits theorem \cite{erdos-simonovits, erdos-stone}. While their result also holds for bipartite $H$, it only tells us that $\mbox{ex}(n,H) = o(n^2)$, which is less satisfactory since stronger estimates exist. For instance, the K{\H o}vari--S{\'o}s--Tur{\'a}n theorem \cite{kovari} states that for the complete bipartite graph $K_{s,t}$ we have $\mbox{ex}(n,K_{s,t})<c_{s,t}n^{2-\frac{1}{s}}$. There are many long-standing unsolved questions in this area and we refer the reader to the extensive survey \cite{furedi} for more information and further references. 

\medskip

Recently, Loh, Tait, Timmons, and Zhou \cite{loh} introduced a new and natural line of investigations related to the Tur{\'a}n type problems. For a pair of graphs $H$ and $F$, they proposed the problem of determining the maximum number of edges in a graph on $n$ vertices, subject to the condition that we simultaneously forbid $H$ as a subgraph and $F$ as an {\em induced} subgraph. One of their main results \cite[Theorem 1.1]{loh} addresses the case when $H=K_r$ and $F = K_{s,t}$, where they obtain the same asymptotic upper bound as in the K{\H o}vari-S{\'o}s-Tur{\'a}n theorem (with a different constant, now depending on $r$, $s$, and $t$). The case $s=t=2$ is interesting in its own right, and is closely related to the following result due to Gy{\'a}rf{\'a}s, Hubenko, and Solymosi \cite{ghs}.

\begin{theorem*}[Gy{\'a}rf{\'a}s-Hubenko-Solymosi]\label{GHSthm}
Let $G$ be a graph on $n$ vertices and at least $c\binom{n}{2}$ edges. If $G$ does not contain $K_{2,2}$ as an induced subgraph, then $\omega(G)\geq \frac{c^2}{10}n$.
\end{theorem*}

Here $\omega(G)$ denotes the maximum number of vertices in a clique contained in $G$. In the aforementioned paper by Loh {\em et al.}, they also give an extension the Gy{\'a}rf{\'a}s--Hubenko--Solymosi Theorem to the case when the forbidden induced subgraph is $K_{2,t}$ \cite[Theorem 1.3]{loh}.

\medskip 

The goal of this paper is to extend the Gy{\'a}rf{\'a}s--Hubenko--Solymosi theorem in another direction, more specifically, for $k$-uniform hypergraphs. Throughout the paper we use the following standard notation and terminology. For a positive integer $k$ we let $[k]$ denote the set $\{1,\dots, k\}$. For a finite set $S$ we let $\binom{S}{k}$ denote the set of all $k$-tuples (i.e. $k$-element subsets) of $S$. A {\em $k$-uniform hypergraph} $H=(V,E)$ consists of a finite set of vertices $V$ and a set of edges $E\subset \binom{V}{k}$. A subset $S\subset V$ forms a {\em clique} in $H$ if $\binom{S}{k}\subset E$, and we let $\omega(H)$ denote the maximum number of vertices in a clique in $H$.

In order to avoid using ceiling and floor functions in calculations, we extend the binomial coefficient as the continuous convex function
\[{\binom{x}{k}} = \begin{cases} \frac{x(x-1)\cdots(x-k+1)}{k!} & x\geq k-1 \\ 0 & x<k-1.\end{cases}\] 

\subsection*{Results} We start by giving a new proof of the  Gy{\'a}rf{\'a}s--Hubenko--Solymosi theorem which has the advantage of producing a quantitative improvement. 

\begin{theorem} \label{1improvement}
Let $G$ be a graph on $n$ vertices and at least $\alpha\binom{n}{2}$ edges. If $G$ does not contain $K_{2,2}$ as an induced subgraph, then $\omega(G) \geq (1-\sqrt{1-\alpha})^2n$.
\end{theorem}

{\em Remark.} It is interesting to note that if $G$ is a chordal graph on $n$ vertices and $\alpha\binom{n}{2}$ edges, then $\omega(G)\geq (1-\sqrt{1-\alpha})n$, which is best possible. (This is a result due to Katchalski and Abbot \cite{katch-abb}, and was also shown in \cite{ghs}.) The appearance of the factor $(1-\sqrt{1-\alpha})$ in our new bound seems to be a coincidence, and the problem of determining the optimal linear factor in the general case of no induced $K_{2,2}$, even for specific values of $\alpha$, remains an open, although some progress has been made in \cite{gyar-sark}.  

\medskip

The main advantage of our new proof of the  Gy{\'a}rf{\'a}s--Hubenko--Solymosi theorem is that it can be extended to $k$-uniform hypergraphs. This is interesting because it has implications in discrete geometry and combinatorial topology, more specifically, with respect to the {\em colorful} and {\em fractional} versions of Helly's theorem \cite{col-hell, eckhoff, floy, helly, kalai-upper, top-col-hel, katch-liu}. This connection will be discussed further in Section \ref{sec:frac}.

\medskip

Let $H=(V,E)$ be a $k$-uniform hypergraph. We call the set $M = \binom{V}{k}\setminus E$ the set of {\em missing edges} of $H$. The following definition extends the notion of an induced $K_{2,2}$ in a graph in several ways.

\begin{definition*} \label{complete set}
Let $H$ be a $k$-uniform hypergraph and $m\geq k$ an integer. A family $\{\tau_1, \dots, \tau_m\}\subset M$ is called a {\em complete $m$-tuple of missing edges} if 
\begin{enumerate}
    \item \label{disjointness} $\tau_i\cap \tau_j = \emptyset$ for all $i\neq j$, and
    \item \label{transversal} $\{t_1, \dots, t_m\}$ is a clique in $H$ for all $t_i\in \tau_i$ and all $i\in [m]$.
\end{enumerate} 
\end{definition*} 

{\em Remark.} Note that for $m=k$, condition (\ref{transversal}) simply says that $\{t_1, \dots, t_k\}$ is an edge in $H$ for every choice $t_1\in \tau_1, \dots, t_k\in \tau_k$. In the case of graphs ($k=2$), a complete $m$-tuple of missing edges is equivalent to an induced $K_2(m)$, that is, the complete multipartite graph on $m$ vertex classes each of size two. However, for $k>2$, a complete $m$-tuple of missing edges should not be thought of as an induced hypergraph since the definition only speaks about edges containing at most one vertex from each $\tau_i$.

\medskip

For a $k$-uniform hypergraph $H$ and $m\geq k$, let $c_m(H)$ denote the number of cliques on $m$ vertices in $H$. In particular, $c_k(H)$ denotes the number of edges in $H$. We may now state our main result.

\begin{theorem} \label{k-uniform thm}
For any $m\geq k\geq 2$ and $\alpha \in (0,1)$, there exists a constant $\beta = \beta(\alpha, k , m) > 0$ with the following property:
Let $H$ be a $k$-uniform hypergraph on $n$ vertices and $c_m(H) \geq \alpha\binom{n}{m}$. If $H$ does not contain a complete $m$-tuple of missing edges, then $\omega(H)\geq \beta n$. 
\end{theorem}

\subsection*{Outline of paper} In section \ref{sec:newproof} we give the new proof of the Gy{\'a}rf{\'a}s-Hubenko-Solymosi theorem. This proof contains all the main ideas needed for establishing Theorem \ref{k-uniform thm}, which will be done in section \ref{sec:hyper}. Finally, in section \ref{sec:frac} we review the (topological) colorful Helly theorem and the (topological) fractional Helly theorem and show how these are related via Theorem \ref{k-uniform thm}.

\section{Improving the Gy{\'a}rf{\'a}s-Hubenko-Solymosi theorem}\label{sec:newproof}

Here we prove Theorem \ref{1improvement}.
Let $G=(V,E)$ be a graph with $|V|=n$ and $|E|\geq \alpha\binom{n}{2}$. Recall that the missing edges are the elements of $M = \binom{V}{2}\setminus E$.

Let us suppose $\omega(G)\leq \beta n$, where $\beta = (1-\sqrt{1-\alpha})^2$, and that $G$ does not contain $K_{2,2}$ as an induced subgraph. Notice that for our choice of $\beta$, we have
\[(\alpha - \beta) = 2\sqrt{1-\alpha}\sqrt{\beta}.\]

We start by fixing a vertex $v$ and making some observations about its neighborhood $N_v=\{u\in V : uv\in E\}$ and the induced subgraph $G_v= G[N_v]$. The assumption that $G$ does not contain an induced $K_{2,2}$ implies that for every pair of (vertex) disjoint missing edges $\{\bar{e}_1,\bar{e}_2\}$ in $G_v$ there exists another missing edge $\bar{e}_3$ which has one vertex in common with $\bar{e}_1$ and one with $\bar{e}_2$. Letting
$m_v$ denote the total number of missing edges in $G_v$, and $\mu_v$ denote the maximum number of pairwise disjoint missing edges in $G_v$, we obtain
\[m_v\geq \mu_v + \binom{\mu_v}{2} = \binom{\mu_v+1}{2}.\]
Note that the vertices in $G_v$ not covered by a maximal mathcing of missing edges must form a clique and that $\omega(G_v)\leq \beta n-1$, which implies \[\mu_v+1\geq  \frac{|N_v|-\beta n+3}{2}.\]

Summing over all $v\in V$ and using $\sum \frac{|N_v|}{2}= |E| \geq \alpha\binom{n}{2}$, we have

\begin{align*}
\sum (\mu_v+1) & \geq \alpha\binom{n}{2} -\beta\frac{n^2}{2} + \frac{3n}{2}\\ 
 &= (\alpha - \beta)\frac{n^2}{2}+(3-\alpha)\frac{n}{2} \\ 
 &\geq \sqrt{1-\alpha}\sqrt{\beta}n^2 + n.
\end{align*}

By Jensen's inequality, we get

\begin{align*}
\sum m_v &\geq \sum \binom{\mu_v+1}{2}\geq
n\binom{\frac{1}{n}\sum(\mu_v+1)}{2}\\
& \geq n\binom{\sqrt{1-\alpha}\sqrt{\beta}n+1}{2} \\ 
& \geq (1-\alpha)\binom{n}{2}\beta n.
\end{align*}

Since the total number of missing edges in $G$ is at most $(1-\alpha)\binom{n}{2}$, by the pigeon-hole principle, there is a missing edge $\bar{e}$ and a subset of vertices $S\subset V$, with $|S|\geq \beta n$, such that $\bar{e}$ is in the neigborhood of every vertex in $S$. There can not be a missing edge contained in $S$, since together with $\bar{e}$ this would form an induced $K_{2,2}$. Therefore $S$ forms a clique which implies $\omega(G)\geq \beta n$. \qed

\section{Extending to hypergraphs}\label{sec:hyper}

We start by generalizing the two key steps from the proof in the previous section. The case $m=k$ of the following lemma was used implicitly in the proof of \cite[Theorem 2.2]{minki}.

\begin{lemma}\label{missing k-tuples}
Let $H=(V,E)$ be a $k$-uniform hypergraph and let $m\geq k$. If $H$ does not contain a complete $m$-tuple of missing edges, then any subset $S\subset V$ contains at least $\binom{m}{k}^{-1}\binom{\frac{1}{k}(|S|-\omega(H))}{k}$ missing edges of $H$.
\end{lemma}

\begin{proof}
We may assume $|S|>\omega(H)$, otherwise there is nothing to prove, so therefore $S$ contains at least one missing edge. Let $\tau_1, \dots, \tau_t$ be a maximal matching of missing edges contained in $S$. Since $S\setminus(\tau_1\cup \cdots \cup \tau_t)$ contains no missing edges, we have

\[kt = |\tau_1\cup \cdots \cup \tau_t| \geq |S|-\omega(H).\] 

Using the hypothesis that $H$ has no complete $m$-tuple of missing edges, it follows that for every $I\in \binom{[t]}{m}$ there is $J\in \binom{I}{k}$ and a missing edge $\tau \in M$ such that $|\tau\cap \tau_j| = 1$ for all $j\in J$. Since this particular missing edge can appear in this way for at most $\binom{t-k}{m-k}$ distinct $m$-tuples of $[t]$, it follows that $S$ contains at least 
\begin{align*}
\frac{\binom{t}{m}}{\binom{t-k}{m-k}} = \frac{\binom{t}{k}}{\binom{m}{k}}
\end{align*}
distinct missing edges. 
\end{proof}

In the proof of Theorem \ref{k-uniform thm} we will iteratively build up a set of missing edges (eventually ending up in a complete $m$-tuple of missing edges or a clique). This iterative process is defined by the following.

\begin{lemma} \label{iterated function}
Let $H=(V,E)$ be a $k$-uniform hypergraph on $n$ vertices with $\omega(H) \leq (c/2)n$, and suppose $H$ does not contain a complete $m$-tuple of missing edges. The following holds for any $i\geq 2$ and for all sufficiently large $n$:
Given a family $F_i \subset \binom{V}{i}$ with $|F_i|\geq c\binom{n}{i}$, there exists a family $F_{i-1}\subset \binom{V}{i-1}$ and a missing edge $\tau$ of $H$ such that 
\begin{enumerate}
    \item $|F_{i-1}|\geq \left(\frac{c}{12km}\right)^k \binom{n}{i-1}$, and  
    \item $\sigma\cup \{t\}\in F_i\:$ for all $\sigma\in F_{i-1}$ and all $t\in \tau$.
\end{enumerate}
\end{lemma}

\begin{proof}
For every $\sigma\in \binom{V}{i-1}$ define the set 
$N_\sigma = \{x\in V : \sigma\cup \{x\}\in F_{i}\}$. 
We want to lower bound the size of the set
\[X = \{(\sigma, \tau) : \sigma\in \textstyle\binom{V}{i-1}, \tau\in M\cap \textstyle{\binom{N_\sigma}{k}}\}.\]

By Lemma \ref{missing k-tuples} and Jensen's inequality, we get

\begin{align*}
|X| & \geq \binom{m}{k}^{-1} \sum_{\sigma\in\binom{V}{i-1}} \binom{\frac{1}{k}(|N_\sigma|-(c/2)n)}{k} \\
& \geq 
\binom{m}{k}^{-1}
\binom{n}{i-1}
\binom{\binom{n}{i-1}^{-1}\frac{1}{k}\sum_{\sigma\in\binom{V}{i-1}}(|N_\sigma|-(c/2)n)}{k}.
\end{align*}

Using  $\sum_{\sigma\in\binom{V}{i-1}}|N_\sigma| = i|F_i| \geq ic\binom{n}{i}$ and $\binom{n}{i} > \frac{(n-i)}{i}\binom{n}{i-1}$, we get

\begin{align*}
|X| & >
\binom{m}{k}^{-1}
\binom{n}{i-1}\binom{\frac{c}{2k}n - \frac{ci}{k}}{k}. 
\end{align*}

Since the term $\frac{ci}{k}$ is a constant which does not depend on  $n$, it follows that for all sufficiently large $n$ (depending only on $c, i$, and $k$), we get 

\begin{align*}
|X| & > \frac{1}{2}\binom{m}{k}^{-1}\binom{n}{i-1}\binom{\frac{c}{2k}n}{k} 
\geq \left(\frac{c}{12km}\right)^k\binom{n}{i-1}\binom{n}{k}.
\end{align*}

By averaging, there exists a missing edge $\tau\in M$ such that $\tau\subset N_\sigma$ for at least $\frac{|X|}{|M|}$ distinct $\sigma\in \binom{[n]}{i-1}$. The lemma now follows since $|M|\leq \binom{n}{k}$.
\end{proof}

\bigskip

\begin{proof}[Proof of Theorem \ref{k-uniform thm}] Let $f(x) = (\frac{x}{12km})^k$. (Note that $0<f(x)<x/2$ for all $x\in(0,1)$.) Define $\alpha_0 = \alpha$ and $\alpha_i = f(\alpha_{i-1})$ for all $i\geq 1$. We will show the theorem holds with $\beta = \beta(\alpha,k,m) = \alpha_{m-1}>0$.

\medskip

Let $F_m\subset \binom{V}{m}$ be the the set of $m$-tuples that form cliques in $H$, and so by hypothesis, we have $|F_m| = c_m(H) \geq \alpha_0\binom{n}{m}$. Assuming both $\omega(H) \leq \beta n$ and that $H$ does not contain a complete $m$-tuple of missing edges, we can apply Lemma \ref{iterated function} iteratively, starting with $F_m$, obtaining a family $F_{m-1}$, to which we apply Lemma \ref{iterated function}, and so on. Moreover, at each step we pick up a new missing edge. 

For every $1 \leq i < m$, we claim that after the $i$th application of Lemma \ref{iterated function} we have obtained a subfamily 
$F_{m-i}\subset \binom{V}{m-i}$ and 
pairwise disjoint missing edges $\tau_1, \dots, \tau_i$ such that 
\begin{itemize}
    \item $|F_{m-i}|\geq \alpha_{i}\binom{n}{m-i}$, and
    \item $\sigma \cup \{t_1, \dots, t_i\}\in F_m$ for all $\sigma\in F_{m-i}$ and all $t_1\in \tau_1, \dots, t_i\in \tau_i$.
\end{itemize}
The claim is true for $i=1$, as this is just the statement of Lemma \ref{iterated function}. Assuming it is true for some $i$, we now check that it holds for $i+1$. 

Applying Lemma \ref{iterated function} to $F_{m-i}$, we obtain a family $F_{m-i-1} \subset \binom{V}{m-i-1}$ and a missing edge $\tau_{i+1}$ such that 
\begin{itemize}
    \item $|F_{m-i-1}|\geq \alpha_{i+1}\binom{n}{m-i-1}$, and 
    \item $\sigma \cup \{t_{i+1}\}\in F_{m-i}$ for all $\sigma\in F_{m-i-1}$ and all $t_{i+1}\in \tau_{i+1}$.
\end{itemize}
But our assumption on $F_{m-i}$ therefore implies that $\sigma \cup \{t_1, \dots, t_{i+1}\}\in F_m$ for all $\sigma\in F_{m-i-1}$ and $t_1\in \tau_1, \dots, t_{i+1}\in \tau_{i+1}$. Note that this also implies that $\tau_{i+1}$ must be disjoint from every $\tau_1, \dots, \tau_i$. This proves the inductive step.

After $m-1$ applications of Lemma \ref{iterated function}, we end up with a subset $F_1\subset V$ and pairwise disjoint missing edges $\tau_1, \dots, \tau_{m-1}$ such that 
\begin{itemize}
    \item $|F_1|\geq \alpha_{m-1} n = \beta n$, and
    \item $\{t\}\cup \{t_1, \dots, t_{m-1}\}\in F_m$ for all $t\in F_1$ and all $t_1\in \tau_1, \dots, t_{m-1}\in \tau_{m-1}$.
\end{itemize}

If $F_1$ contains a missing edge $\tau_m\in M$, then $\{\tau_1, \dots, \tau_m\}$ would be a complete $m$-tuple of missing edges in $H$. Since we assumed this does not exist, it follows that $F_1$ is a clique in $H$, and so $\omega(H)\geq \beta n$.
\end{proof}

{\em Remark.} In the proof above, Lemma \ref{iterated function} was used $m-1$ times, and it follows that for fixed $k$ and $m$ we have $\beta= \Omega(\alpha^{k^{m-1}})$. If we consider the optimal function $\beta = \beta(\alpha, k, m)$ for which Theorem \ref{k-uniform thm} holds, it is worth noting that $\beta\to 1$ as $\alpha \to 1$. This does not follow from our definition of $\beta$ in the proof above, but can be deduced directly from Lemma \ref{missing k-tuples} by setting $S=V$. The lemma then tells us that if $\omega(H)\leq (1-\epsilon)n$, then $H$ has at least $\binom{m}{k}^{-1}\binom{\frac{\epsilon}{k}n}{k}$ missing edges. It is easy to show by a simple double-counting argument that this implies $c_m(H)\leq (1-\delta)\binom{n}{m}$ for some absolute $\delta>0$.

\section{Applications} \label{sec:frac}

Here we present some applications of Theorem \ref{k-uniform thm} related to certain Helly-type theorems and the intersection patterns of convex sets. For more information about these types of results we refer the reader to the surveys \cite{ADLS, eckhoff-surv}.

Helly's theorem \cite{helly} asserts that if every $d+1$ members of a finite family of convex sets in $\mathbb{R}^d$ have a point in common, then there is a point in common to every member of the family. Among the numerous generalizations and extensions of Helly's theorem we focus on two important generalizations. The first one is the {\em Colorful Helly Theorem} discovered by Lov{\'a}sz and reported by B{\'a}r{\'a}ny in \cite{col-hell}.

\begin{theorem*}[Colorful Helly]
Let $F_1, \dots, F_{d+1}$ be finite families of convex sets in $\mathbb{R}^d$. Suppose for every choice $K_1\in F_1$, $\dots$, $K_{d+1}\in F_{d+1}$ we have $\bigcap_{i=1}^{d+1}K_i\neq \emptyset$. Then for one of the families $F_i$ we have $\bigcap_{K\in F_i}K\neq \emptyset$.
\end{theorem*}
Note that we recover Helly's theorem in the case when $F_1=\cdots = F_{d+1}$. The second generalization of Helly's theorem we are interested in is the {\em Fractional Helly Theorem} due to Katchalski and Liu \cite{katch-abb}. (See also \cite[Chapter 8]{mato}.)

\begin{theorem*}[Fractional Helly]
For every $d\geq 1$ and $\alpha\in(0,1)$ there exists a $\beta=\beta(\alpha,d)\in(0,1)$ with the following property:
Let $F$ be a family of $n>d+1$ convex sets in $\mathbb{R}^d$ and suppose at least $\alpha\binom{n}{d+1}$ of the $(d+1)$-tuples in $F$ have non-empty intersection. Then there exists some $\beta n$ members of $F$ whose intersection is non-empty.
\end{theorem*}

Our first application is a new proof of the Fractional Helly Theorem, which  uses the Colorful Helly Theorem and Theorem \ref{k-uniform thm}.

\begin{proof}[Proof of the fractional Helly theorem]  Define a $(d+1)$-uniform hypergraph $H = (F,E)$ where $E = \{\sigma\in \binom{F}{d+1} :\bigcap_{K\in\sigma} K \neq \emptyset\}$.  By hypothesis, $H$ has at least $\alpha\binom{n}{d+1}$ edges, and by the Colorful Helly Theorem $H$ does not contain a complete $(d+1)$-tuple of missing edges. So by Theorem \ref{k-uniform thm}, with $k=m=d+1$, there exists a $\beta>0$ such that $H$ has a clique on $\beta n$ vertices. By Helly's theorem, the members of  $F$ contained in this clique have non-empty intersection.    
\end{proof}

The argument above is general enough to give a proof of a topological generalization of the fractional Helly theorem proved by Kalai \cite{kalai-upper}, and independently by Eckhoff \cite{eckhoff} in a slightly restricted setting. 

Let $K$ be a finite simplicial complex. For an integer $i\geq 0$, let $f_i(K)$ denote the number of $i$-dimensional faces in $K$.
We say that $K$ is {\em $d$-Leray} if $\tilde{H}_i(X) = 0$ for all induced subcomplexes $X\subset K$ and all $i\geq d$. (Here $\tilde{H}_i(X)$ denotes the $i$-dimensional homology of $X$ with coefficients in $\mathbb{Q}$.)

The following is a consequence of the ``upper-bound theorem'' for $d$-Leray complexes due to Kalai \cite{kalai-upper}, and implies the Fractional Helly Theorem (via the Nerve theorem, see e.g. \cite[Corollary 4G.3]{hatcher}).

\begin{theorem*}[Topological Fractional Helly] \label{top-frac-helly} For every $d\geq 1$ and $\alpha\in (0,1)$ there exists a $\beta=\beta(\alpha,d)\in (0,1)$ with the following property:
If $K$ is a $d$-Leray complex with $f_0(K)=n$ and $f_d(K)\geq \alpha\binom{n}{d+1}$, then $K$ has dimension at least $\beta n-1$.
\end{theorem*}

Kalai's proof of this result relies on his technique of algebraic shifting. (See also \cite[Section 6]{AKKM} for other algebraic approaches.)
We want to give a proof of the Topological Fractional Helly Theorem using Theorem \ref{k-uniform thm}, but first we need the following auxiliary result, due to Kalai and Meshulam \cite[Theorem 1.6]{top-col-hel}. (See also \cite{floy} for an algebraic generalization.)

\begin{theorem*}[Topological Colorful Helly]
Let $X$ be a $d$-Leray complex on the vertex set $V$ and let $M$ be a matroidal complex on $V$ such that $M\subset X$. Then there exists a simplex $\tau\in X$ such that $\rho(V\setminus \tau)\leq d$.
\end{theorem*} (Here $\rho$ denotes the rank function of the matroid $M$.)
Let us describe the special case of the Topological Colorful Helly Theorem that we  need. Let $V_1, \dots, V_{d+1}$ be distinct finite sets with $|V_1| = \cdots = |V_{d+1}|= d+1$, and let $V = V_1 \cup \cdots \cup V_{d+1}$. Define the simplicial complex $M_{d+1}$ as the join of the $V_i$, that is
\[M_{d+1} = V_1 * \cdots * V_{d+1} = \{\sigma\subset V : |\sigma \cap V_i|\leq 1, \text{ for all } i\}.\]
Note that $M_{d+1}$ is the matroidal complex of the partition matroid induced by the $V_i$, which has rank $d+1$. 

Suppose $X$ is a $d$-Leray complex such that $M_{d+1}\subset X$. By the Topological Colorful Helly Theorem, there is face $\tau \in X$ such that $\rho(X\setminus \tau)\leq d$. But this means that for some $i$ we have $V_i\subset \tau$, so in particular one of the $V_i$ is a face in $X$.

\begin{proof}[Proof of the Topological Fractional Helly Theorem.]
Let $H=(V,E)$ be the $(d+1)$-uniform hypergraph where $V$ is the vertex set of $K$ and $E$ is the set of $d$-dimensional faces of $K$. Thus, $H$ has $n$ vertices and at least $\alpha\binom{n}{d+1}$ edges. In order to apply Theorem \ref{k-uniform thm}, with $k=m=d+1$, we need to show that $H$ does not contain a complete $k$-tuple of missing edges. But this is precisely the special case of the Topological Colorful Helly Theorem we described above. 

Now Theorem \ref{k-uniform thm} implies that there is a clique in $H$ on at least $\beta n$ vertices, which corresponds to a subcomplex $K'\subset K$ on at least $\beta n$ vertices whose $d$-dimensional skeleton is complete. The $d$-Leray property now implies that $K'$ is a full simplex. 
\end{proof}

{\em Remark.}
It should be noted that Kalai's ``upper-bound theorem'' actually implies the Topological Fractional Helly Theorem with $\beta = 1-(1-\alpha)^{\frac{1}{d+1}}$, which is best possible. Our proof gives a far weaker bound on $\beta$, but this is not surprising since the $d$-Leray property is much stronger than excluding a complete set of missing edges in $H$. 
It would be interesting to find examples for the hypergraph setting of Theorem \ref{k-uniform thm} which give non-trivial upper bounds on $\beta$.

\medskip

{\em Acknowledgements.} The author thanks Xavier Goaoc and Seunghun Lee for pointing out some mistakes in an earlier version of this manuscript.

\end{document}